\documentclass{amsart}

\usepackage{graphicx}
\usepackage{geometry}
\usepackage{amssymb}
\usepackage{amsmath}
\usepackage{amsfonts}
\usepackage{enumitem}
\usepackage{comment}
\usepackage{overpic}
\usepackage{amsaddr}
\usepackage{hyperref}
\usepackage[nameinlink]{cleveref}

\newcommand{\R}{\mathbb R}

\renewcommand{\phi}{\varphi}

\newtheorem{theorem}{Theorem}[section]

\newtheorem{corollary}{Corollary}[section]
\newtheorem{lemma}[theorem]{Lemma}
\theoremstyle{definition}

\theoremstyle{remark}

\numberwithin{equation}{section}
\crefname{equation}{}{}

\usepackage{xcolor}
\definecolor{matlabblue}{rgb}{0,0.4470,0.7410}
\definecolor{matlabred}{rgb}{0.8500,0.3250,0.0980}
\definecolor{matlabyellow}{rgb}{0.9290,0.6940,0.1250}

\usepackage[obeyFinal,textsize=footnotesize]{todonotes}

\title{Convergence of pivoted Cholesky for Lipschitz Kernels}

\author{Sungwoo Jeong, Alex Townsend, and Christopher Wang}
\email{sjeong@cornell.edu, townsend@cornell.edu, cyw33@cornell.edu}
\address{Department of Mathematics, Cornell University}
\date{}

\begin{document}
\begin{abstract}
    We investigate the continuous analogue of the Cholesky factorization, namely the pivoted Cholesky algorithm. Our analysis establishes quantitative convergence guarantees for kernels of minimal smoothness. We prove that for a symmetric positive definite Lipschitz continuous kernel $K:\Omega\times \Omega \rightarrow \mathbb{R}$ on a compact domain $\Omega\subset\mathbb{R}^d$, the residual of the Cholesky algorithm with complete pivoting (where the maximum value of the diagonal of the residual is selected as the next pivot) is uniformly bounded above by a constant multiple of the separation distance of the pivots, which implies the convergence rate $\|R_n\|_\infty=\mathcal{O}(n^{-1/d})$, where $R_n$ is the residual after $n$ Cholesky steps and $\|\cdot\|_\infty$ is the supremum norm. Along the way, we show that for any pivoting strategy, $\|R_n\|_\infty$ is bounded by a constant multiple of the fill distance of the pivots. Moreover, if $K$ is differentiable in both variables with a Lipschitz derivative and $\Omega$ is convex, our convergence rate improves to $\mathcal{O}(n^{-2/d})$. Our result closes a gap between theory and practice as previous analyses required $K$ to be $C^2$-regular or translation invariant to establish convergence, whereas empirical evidence indicated robust performance even for non-differentiable kernels. We further prove a connection between the Cholesky algorithm and the P-greedy interpolation method, as well as downstream applications to Gaussian process regression and matrix Cholesky factorization.
\end{abstract}

\maketitle

\section{Introduction}
The Cholesky decomposition is one of the cornerstone factorizations in numerical linear algebra, providing an efficient and stable way to represent a symmetric positive definite (SPD) matrix as a product of a lower-triangular matrix and its transpose. It underlies fast solvers for linear systems, preconditioners, and algorithms across statistics, optimization, and scientific computing. While the discrete case has been studied extensively, a natural extension to SPD kernels is studied far less. The Cholesky algorithm with a pivoting strategy constructs successive low-rank approximations of a kernel by iteratively selecting pivot points and adding rank-one updates (see~\cref{eq:CholeskyAlgorithm}). Despite the pivoted Cholesky algorithm having many downstream applications and related variants such as adaptive cross approximation (ACA)~\cite{bebendorf2000approximation}, pseudoskeleton approximation~\cite{goreinov1997theory}, and P-greedy interpolation~\cite{de2005near}, its convergence theory is not fully developed.

The pivoted Cholesky algorithm is best viewed as an iterative low-rank approximation scheme, rather than as a factorization. Formally, let $\Omega\subset\mathbb{R}^d$ be a compact domain, by which we mean that $\Omega$ is non-empty, connected, and compact. Given a SPD kernel\footnote{A kernel is called SPD if for any finite set of points $x_1,\ldots,x_n \in \Omega$, the matrix $\big[K(x_i,x_j)\big]_{i,j=1}^n$ is SPD.} $K:\Omega\times \Omega \rightarrow \mathbb{R}$, the first step of the pivoted Cholesky algorithm sets $R_0 = K$ and selects a point $z_1$ (which we call a pivot) based on prescribed selection criteria, which we call the \textit{pivoting strategy}. Then, the algorithm constructs a rank-$1$ approximation of $K$ as
\begin{equation*}
K_1(x,y) = \frac{R_0(x,z_1)R_0(z_1,y)}{R_0(z_1,z_1)},
\end{equation*}
which interpolates $K$ whenever $x = z_1$ or $y = z_1$ and the first residual $R_1=K - K_1$ is obtained. We iteratively apply this procedure. At the $n\geq 1$ step, we select a pivot point $z_n$ according to the pivot strategy and compute the rank-$n$ approximant of $K$ by adding a rank-$1$ kernel to a previous approximant, i.e., 
\begin{equation}
    K_n(x, y) = K_{n-1}(x, y) + \frac{R_{n-1}(x, z_n)R_{n-1}(z_n, y)}{R_{n-1}(z_n, z_n)},
    \label{eq:CholeskyAlgorithm}
\end{equation}
with the residual $R_n=K - K_n$. The approximant $K_n$ interpolates $K$ whenever $x = z_i$ or $y = z_i$, for $i\leq n$. The choice of pivoting strategy is crucial and often determines the quality of the approximation. Various pivoting strategies such as complete pivoting~\cite{golub2013matrix} (where $z_n = \arg\max_{x\in\Omega} |R_{n-1}(x,x)|$), maximum volume pivoting~\cite{gu1996efficient,hong1992rank} (where $z_n$ maximizes the absolute determinant of a certain submatrix), and rook pivoting~\cite{neal1992geometric} are commonly used in practice. Since $K$ is positive definite, $R_n$ is positive semidefinite for all $n$, $R_n(x,x)\geq 0$ for all $x\in\Omega$, and $\|R_n\|_\infty = \max_{x,y\in\Omega} |R_n(x,y)|$~\cite{townsend2015continuous}.

The central question is then: How well does $K_n$ approximate $K$ as $n$ grows? Classical theory provides an answer for the best-case scenario. As proved by Schmidt~\cite{schmidt1907theorie} in 1907, the accuracy of any rank-$n$ approximation cannot be better than the $(n+1)$st singular value of $K$. Throughout, we assume that $K$ is a square-integrable kernel, so that we have
\begin{equation*}
     \operatorname{vol}(\Omega) \|R_n\|_{\infty} \geq \|R_n\|_2 \geq \sigma_{n+1}(K),
\end{equation*}
where $\sigma_{n+1}(K)$ denotes the $(n+1)$st singular value of $K$, $\|\cdot\|_2$ is the operator (spectral) norm of the integral operator induced by the kernel $R_n$, and $\|\cdot\|_{\infty}$ is the absolute maximum value on $\Omega\times \Omega$. For example, take the Brownian kernel $K : [0,1]^d\times [0,1]^d\rightarrow \mathbb{R}$ given by $K(x,y) = \prod_{i=1}^d\min\{x_i,y_i\}$. Then, $\sigma_{n}(K) \geq C_d \log^{2(d-1)}(n)/n^2$ for some constant $C_d$ that only depends on $d$, so $\|R_n\|_\infty$ can at best converge to zero at that rate.  In general, one can show via Legendre polynomial approximation~\cite{wang2023new} that if $K\in C^{k,1}(\Omega\times\Omega)$, i.e., $K$ is $k$-times continuously differentiable with Lipschitz $k$th partial derivatives, then $\sigma_{n+1}(K) = \mathcal O(n^{-(k+1)})$ (see e.g.,~\cite[Sec.~A.3]{wang2025operator}).

Bounding $\|R_n\|_\infty$ from above has long been studied for smooth kernels. It has been shown that $\|R_n\|_\infty$ decays exponentially fast to zero as $n\rightarrow\infty$ when $d = 1$, the kernel has an analytic extension to a Bernstein ellipse, and complete pivoting is used~\cite{cortinovis2020maximum,townsend2015continuous}. Moreover, when $d=1$, Bebendorf~\cite{bebendorf2000approximation} has an error bound on $\|R_n\|_\infty$ for so-called asymptotically smooth kernels when rook pivoting is employed, though that typically means the kernel must be infinitely differentiable to achieve convergence. Convergence for translation invariant kernels, namely, kernels of the form $K(x,y) = \phi(\|x-y\|)$ for some univariate function $\phi$, is implied by bounds on the power function from the radial basis function (RBF) literature~\cite{santin6convergence}. Another notable result in the RBF literature suggests that the Cholesky algorithm with complete pivoting converges for all kernels $K:\Omega\times \Omega\rightarrow \mathbb{R}$ that are twice continuously differentiable when $\Omega$ is a convex domain~\cite[Thm.~4.3]{de2005near}. 

Unfortunately, these existing results do not explain the convergence of a Cholesky algorithm on kernels with very limited smoothness or without translation-invariant structure that arise throughout computational mathematics and applications. Prominent examples include
\begin{enumerate}
\item Ornstein--Uhlenbeck (exponential) kernels
\[
K(x,y)=\alpha e^{-|x-y|/\ell}, \qquad \alpha,\ell>0,
\]
which is merely Lipschitz on the diagonal;
\item Mat\'ern kernels with low regularity,\footnote{\,Here, $\mathbf{K}_\nu$ is the modified Bessel function of the second kind, $\Gamma$ is the gamma function, and $\|\cdot\|$ denotes the Euclidean distance.}
\[
K_\nu(x,y)=\frac{2^{1-\nu}}{\Gamma(\nu)}\!\left(\frac{\sqrt{2\nu}\,\|x-y\|}{\ell}\right)^\nu
\mathbf{K}_\nu\!\left(\frac{\sqrt{2\nu}\,\|x-y\|}{\ell}\right),\qquad \nu,\ell >0,
\]
where $\nu$ controls the smoothness, and which, for small $\nu$ in particular, is non-differentiable on the diagonal;
\item Brownian-motion (Wiener) kernels on $[0,1]$, given by $K(x,y)=\min\{x,y\}$;
and
\item Green’s functions of uniformly elliptic self-adjoint PDEs, such as
$G(x,y) = -\frac{1}{2\pi}\log\|x-y\|$ for the 2D Laplacian, $G(x,y)=\frac{1}{4\pi\|x-y\|}$ for the 3D Laplacian, or, on $[0,1]$ with Dirichlet data, $G(x,y)=\min\{x,y\}-xy$ for $-u''=f$. 
\end{enumerate}
In practice, the Cholesky algorithm with any reasonable pivoting strategy continues to perform remarkably well for these nonsmooth kernels (see~\Cref{fig:Convergence}).

\begin{figure}[h]
  \centering
  \begin{minipage}{.49\textwidth}
    \centering
    \begin{overpic}[width=\textwidth]{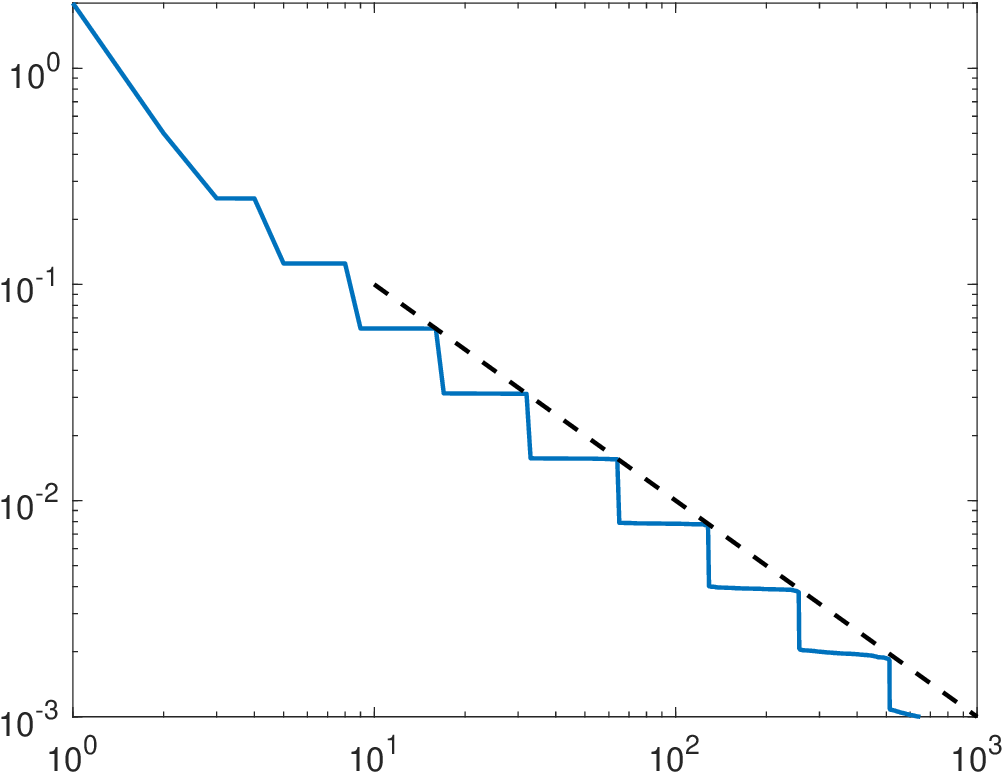}
    \put(60,35) {\rotatebox{-37}{$\mathcal{O}(n^{-1})$}}
    \put(-3,30) {\rotatebox{90}{$\|R_n\|_\infty$}}
    \put(50,-3) {$n$}
    \end{overpic}
  \end{minipage}
  \hfill
  \begin{minipage}{.49\textwidth}
    \centering
    \begin{overpic}[width=\textwidth]{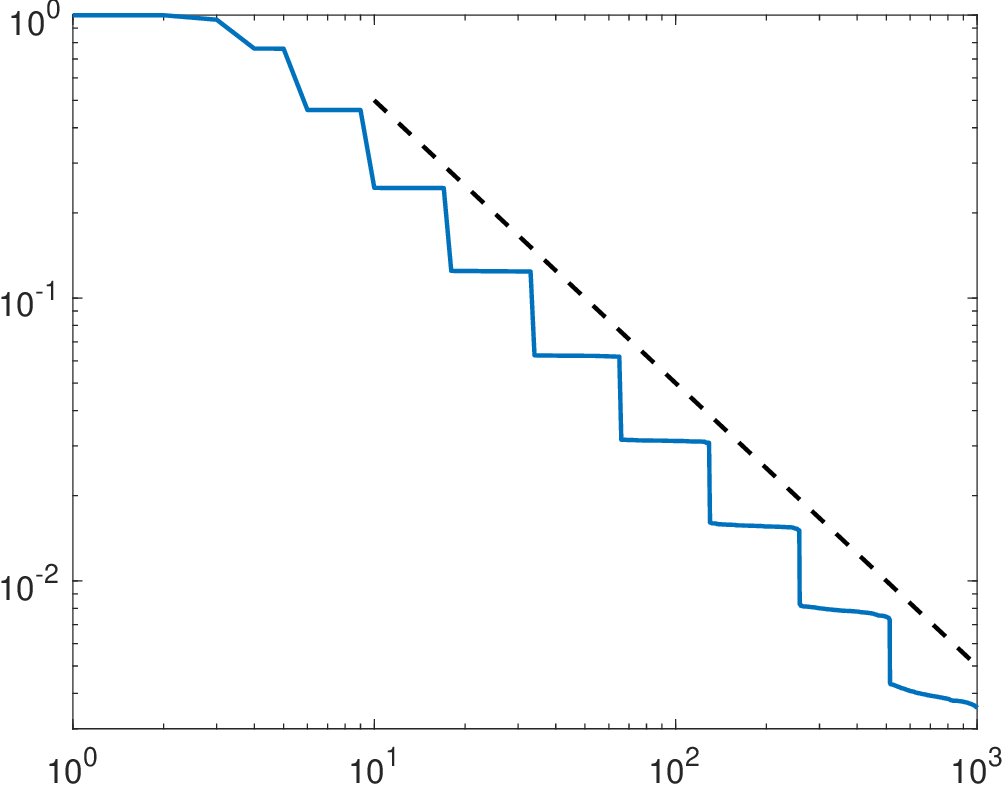}
    \put(65,45) {\rotatebox{-45}{$\mathcal{O}(n^{-1})$}}
    \put(-3,30) {\rotatebox{90}{$\|R_n\|_\infty$}}
    \put(50,-3) {$n$}
    \end{overpic}
  \end{minipage}
  \caption{Convergence results for the Cholesky algorithm with complete pivoting on $\Omega=[-1,1]$ for the Wiener kernel $K(x,y)=\min\{x,y\}+1$ (left) and the Mat\'{e}rn kernel with $\nu = \ell = 1/2$ (right). The $\mathcal O(n^{-1})$ convergence rate matches our predicted rate (see~\Cref{thm:convergencegeneral}) and the best-known rate from the RBF literature (see~\Cref{sec:Pgreedy}) for Lipschitz kernels.}
  \label{fig:Convergence}
\end{figure}

\subsection{Our contributions}
In this paper, we prove error bounds for the pivoted Cholesky algorithm applied to a SPD kernel with Lipschitz continuity along the diagonal, i.e., there is some uniform constant $L>0$ such that
\begin{equation}
|K(x, x) - K(x, y)| \leq L \|x-y\|, \qquad \forall x,y\in\Omega. 
\label{eq:LipschitzAlongDiagonal}
\end{equation} 
This condition is implied by Lipschitz continuity, but is far weaker as it allows the kernel to be discontinuous off the diagonal. The error bound depends on the so-called separation distance of the pivots. If the first $n$ pivots are $Z_n = \left\{z_1,\ldots,z_n\right\}\subset\Omega$, then the separation distance measures the distance between the closest two pivot points and is given by
\begin{equation}\label{eq:sep-dist}
q_{Z_n} := \min_{1\le i<j\le n} \|z_i-z_j\|.
\end{equation}
A related quantity is the fill distance,
\begin{equation}\label{eq:fillDistance}
h_{Z_n} := \sup_{x\in\Omega}\min_{z\in Z_n}\|x-z\|,
\end{equation}
which measures how well the pivot points cover the domain.

\subsection*{Contribution 1}
We prove the following error bound on $\|R_n\|_\infty$ for the Cholesky algorithm with complete pivoting (see~\Cref{thm:convergencegeneral}):
\[
\|R_n\|_\infty \le 4Lq_{Z_n} \le \frac{8LR}{n^{1/d}-1} = \mathcal O(n^{-1/d}),
\]
where $R$ is the radius of a ball containing $\Omega$. The dependence on $q_{Z_n}$ is rather remarkable, since most convergence results for the Cholesky algorithm in the RBF literature are in terms of the fill distance $h_{Z_n}$ instead~\cite{de2005near,wenzel2021novel,wenzel2023analysis}. The problem with these types of estimates is that $h_{Z_n}$ is not guaranteed to decay with the number of pivot points $n$ using complete pivoting absent additional conditions on the kernel such as translation invariance, as pivots may accumulate rather than cover the domain. On the other hand, $q_{Z_n}$ is guaranteed to decay like $O(n^{-1/d})$, which proves convergence. 

Along the way, we recover the bound $\|R_n\|_\infty\le 4Lh_{Z_n}$ in terms of fill distance, which holds for \textit{any} pivoting strategy (see~\Cref{thm:filldistancebound}). Once a particular pivoting strategy is selected for the Cholesky algorithm, we can study if and how $h_{Z_n}\rightarrow0$ as $n\rightarrow \infty$ to derive a rate of convergence.

Additionally, in~\Cref{sec:generalsmooth}, we extend~\Cref{thm:convergencegeneral} to include kernels with derivatives that are Lipschitz on the diagonal, assuming a convex domain, yielding an improved convergence rate of $\|R_n\|_\infty = \mathcal{O}(n^{-2/d})$ with complete pivoting.

\subsection*{Contribution 2}
We derive intriguing matrix counterparts to~\Cref{thm:convergencegeneral}, giving a decay rate for the Cholesky decomposition applied to structured matrices. 
Let $A$ be an $m\times m$ SPD matrix. Then, for $1<n\leq m$, we have (see~\Cref{cor:matrixversion}):
\begin{equation*}
    \|A - A_n\|_{\max} \leq \frac{4(m-1)G_A}{n-1} = \mathcal{O}(n^{-1}), 
\end{equation*}
where $G_A$ is a discrete diagonal Lipschitz constant given by $G_A = \max_{i\neq j}|A_{i, i}-A_{i, j}|/|i-j|$. 

\subsection*{Contribution 3} 
We demonstration a connection between the Cholesky algorithm and the kernel approximation literature. In particular, we show that the Cholesky residual after pivoting at the points $Z_n$ is exactly equal to the square of the so-called power function $P_{Z_n}$ that describes the error of functions that interpolate the kernel at the points $Z_n$~\cite{wenzel2021novel, wenzel2023analysis}. In particular, we show that (see~\Cref{lem:diag-equals-power2})
\[
R_n(x,x) = \big(P_{Z_n}(x)\big)^2.
\]
Consequently, the Cholesky algorithm with complete pivoting is equivalent to the P-greedy algorithm from kernel approximation.

\subsection{Structure of the paper}

The remainder of the paper is organized as follows. In~\Cref{sec:proof} we present our main error bound on the Cholesky algorithm with any pivoting strategy for kernels satisfying~\cref{eq:LipschitzAlongDiagonal}, and then specialize this to complete pivoting to obtain a convergence rate in~\Cref{sec:proofstep2}. In~\Cref{sec:OtherPivotingStrategies}, we discuss other possible pivoting strategies. We improve the convergence rate for $C^{1,1}$ regularity and explain a barrier in our argument for further improvements in~\Cref{sec:generalsmooth}. In~\Cref{sec:connections} we explore the consequences of our results for matrices, Gaussian process regression, and P-greedy interpolation. Finally, we give brief concluding remarks in~\Cref{sec:conclusion}.

\section{Error Bound on the Pivoted Cholesky Algorithm in Terms of Fill Distance}\label{sec:proof}
Initially, proving the convergence of the Cholesky algorithm on Lipschitz kernels seems extremely challenging, since the Lipschitz condition is generally not preserved by a Cholesky step. Suppose one starts with a Lipschitz kernel with constant $L>0$, i.e., 
\[
|K(x_1,y_1) - K(x_2,y_2)| \leq L\left\| \begin{bmatrix}x_1\\y_1\end{bmatrix} - \begin{bmatrix}x_2\\y_2\end{bmatrix}\right\|,\qquad x_1,x_2,y_1,y_2\in\Omega. 
\] 
Since
\begin{equation}\label{eq:residual-1}
R_1(x,y) = K(x,y) - \frac{K(x,z_1)K(z_1,y)}{K(z_1,z_1)},
\end{equation}
where $z_1$ is the first pivot, one finds that even with complete pivoting, one only has  
\[
\begin{aligned}
|R_1(x_1,y_1) - R_1(x_2,y_2)| &= \left|K(x_1,y_1) - \frac{K(x_1,z_1)K(z_1,y_1)}{K(z_1,z_1)} -K(x_2,y_2) + \frac{K(x_2,z_1)K(z_1,y_2)}{K(z_1,z_1)} \right|\\
&\leq L\left\| \begin{bmatrix}x_1\\y_1\end{bmatrix} - \begin{bmatrix}x_2\\y_2\end{bmatrix}\right\| + L\left\| \begin{bmatrix}z_1\\y_1\end{bmatrix} - \begin{bmatrix}z_1\\y_2\end{bmatrix}\right\|,
\end{aligned} 
\]
where in the last inequality we used the triangle inequality and that $|K(x_1,z_1)/K(z_1,z_1)|\leq 1$ and $|K(x_2,z_1)/K(z_1,z_1)|\leq 1$ for complete pivoting. In other words, the Lipschitz condition is not preserved. However, progress can be made for kernels that satisfy the even weaker condition~\cref{eq:LipschitzAlongDiagonal} by focusing on bounding the diagonal of $R_n$ near a pivot location. 

We start by proving that the diagonal of $R_1$ must grow slowly away from a pivot location $z_1$. The following lemma holds for the Cholesky algorithm with any pivoting strategy. 
\begin{lemma}\label{lem:lipschitzchange}
Let $K:\Omega\times \Omega\to\mathbb{R}$ be a continuous SPD kernel satisfying~\cref{eq:LipschitzAlongDiagonal} with constant $L>0$. Then the first residual $R_1$ of the Cholesky algorithm with any pivoting strategy satisfies
\begin{equation*}
    0\leq R_1(x, x) \leq 4L \|x - z_1\|, \qquad x\in\Omega,
\end{equation*} 
where $z_1$ is the first Cholesky pivot.
\end{lemma}
\begin{proof}
First, $K(z_1,z_1)>0$ since $K$ is positive definite. Since $(K(x,z_1)-K(z_1,z_1))^2\ge0$, we have
\[
\frac{K(x,z_1)^2}{K(z_1,z_1)} \ge 2K(x,z_1)-K(z_1,z_1)
\]
which by~\cref{eq:residual-1} yields
\[
|R_1(x,x)|\le |K(x,x)-2K(x,z_1)+K(z_1,z_1)| \le \left|K(z_1,z_1) - K(x,x)\right| + 2|K(x,z_1)-K(x,x)|.
\]
The first term on the right satisfies $|K(z_1,z_1)-K(z_1,x)|+|K(z_1,x)-K(x,x)| \le 2\|x-z_1\|$ by~\cref{eq:LipschitzAlongDiagonal}; the same upper bound holds for the second term on the right. The result follows.
\end{proof}

\Cref{lem:lipschitzchange} tells us that the diagonal of $R_1$ is not only zero at the pivot location $z_1$, but also remains small in a neighborhood around $z_1$. In fact, along the diagonal, $R_1$ cannot grow faster than proportional to the distance away from the pivot. Surprisingly, this property continues to hold as one takes more Cholesky steps, regardless of the pivoting strategy. The main idea is the fact that the Cholesky residual does not depend on the \textit{order} of pivots, which follows from the well-known formula
\begin{equation}\label{eq:cholesky-pseudoinverse}
K_n(x,y) = \big([K(z_i,x)]_{i=1}^n\big)^\top\big([K(z_i,z_j)]_{i,j=1}^n\big)^{-1}[K(z_i,y)]_{i=1}^n,
\end{equation}
where $z_1,\dots,z_n$ are the pivots, $[K(z_i,x)]_{i=1}^n$ and $[K(z_i,y)]_{i=1}^n\in\R^n$ are column vectors in $\R^n$, and $[K(z_i,z_j)]_{i,j=1}^n$ is the $n\times n$ kernel matrix evaluated at the pivots.

\begin{lemma}\label{lem:MultiStepLipschitzControl}
Let $K:\Omega\times \Omega\to\mathbb{R}$ be a continuous SPD kernel satisfying~\cref{eq:LipschitzAlongDiagonal} with constant $L>0$. Then, the $n$th residual $R_n$ of the Cholesky algorithm with any pivoting strategy satisfies
\begin{equation*}
   0\leq R_n(x, x) \leq 4L \min_{1\leq i\leq n} \|x - z_i\|, \qquad x\in\Omega,
\end{equation*} 
where $z_1,\dots,z_n$ are the Cholesky pivots.
\end{lemma} 
\begin{proof} 
Since $K$ is positive definite, $R_n$ is positive semidefinite and thus, $R_n(x,x)\geq 0$ for all $x\in\Omega$. 
Select the pivot location closest to $x$, say $z_i$. Since the residual $R_n$ constructed by a pivoted Cholesky algorithm is invariant to the order of the pivots, we obtain the same residual $R_n$ if we perform Cholesky steps with pivots in the order of $z_i, z_1, \dots, z_{i-1}, z_{i+1}, \dots, z_n$, instead of $z_1, \dots, z_n$. Let $\tilde{R}_1$ be the residual after performing a Cholesky step on $K$ with the pivot $z_i$ first. Then, by~\Cref{lem:lipschitzchange}, we have $\tilde{R}_1(x, x) \leq 4L \|x-z_i\|$, since $\tilde{R}_1(z_i, z_i) = 0$. Performing the remaining $n-1$ Cholesky steps on $\tilde{R}_1$ with pivots $z_1, \dots, z_{i-1}$, $z_{i+1}, \dots, z_n$ we obtain $R_n$. Since the diagonal of the residual is nonincreasing under Cholesky steps, i.e., $R_{k+1}(x,x)\leq R_{k}(x,x)$ for all $x\in\Omega$, we find that $R_n(x, x)\leq \tilde{R}_1(x, x)$ for any $x\in\Omega$.
\end{proof} 

In particular,~\Cref{lem:MultiStepLipschitzControl} tells us that near any pivot location, $R_n$ is small. Or, said more precisely, for any $\epsilon>0$, we know that $R_n(x,x)\leq \epsilon$ provided $\min_{1\leq i\leq n} \|x - z_i\|\leq \epsilon/(4L)$. Once the diagonal of $R_n$ is controlled, one can bound $\|R_n\|_\infty$. We can now prove a bound on $\|R_n\|_\infty$ in terms of the fill distance of the pivots (see~\Cref{thm:filldistancebound}). 

\begin{theorem}\label{thm:filldistancebound}
    Let $K:\Omega\times\Omega \rightarrow\mathbb{R}$ be a SPD kernel such that~\cref{eq:LipschitzAlongDiagonal} holds for some constant $L>0$. Then, the $n$th residual $R_n$ of the pivoted Cholesky algorithm satisfies
\[
        \|R_n\|_{\infty} \leq  4L h_{Z_n},
\]
    where $Z_n=\{z_1, \dots, z_n\}$ is the set of pivots and $h_{Z_n}$ is defined in~\cref{eq:fillDistance}. 
\end{theorem}
\begin{proof}
Let $Z_n = \{z_1,\ldots,z_n\}$ be the pivots selected by Cholesky. Since $R_n(z_i,z_i) = 0$ for $1\leq i\leq n$, we know from~\Cref{lem:MultiStepLipschitzControl} that 
\[
R_n(x,x)\leq 4L \min_{z_i\in Z_n} \|x - z_i\|\leq 4L h_{Z_n}, \qquad x\in\Omega, 
\]
where the last inequality follows directly from the definition of $h_{Z_n}$.  Finally, since $R_n$ is itself positive semidefinite, we know that\footnote{\,This follows because the $2\times 2$ matrix $\begin{bmatrix} R(x,x) & R(x,y) \\ R(y,x) & R(y,y)\end{bmatrix}$ is positive semidefinite.}
\[
|R_n(x,y)| \leq \sqrt{R_n(x,x)R_n(y,y)} \leq 4L h_{Z_n}, \qquad x,y\in\Omega,
\]
as desired.
\end{proof}

Just because $\|R_n\|_\infty\leq 4Lh_{Z_n}$ does not necessarily imply that $\|R_n\|_{\infty}\to0$ as $n\to\infty$, since the fill distance $h_{Z_n}$ need not shrink.   In fact, it is plausible that the pivot points could accumulate in such a way that the fill distance is not reduced as $n\rightarrow \infty$. One pivoting strategy that minimizes the fill distance is uniform pivoting, where the pivot locations are predetermined and selected based on $\Omega$, not $K$. One can select a set of (quasi-)uniform points in $\Omega$ as the pivot locations. If $\Omega=[0,1]^d$ and $Z_n$ is taken as a tensor grid with $m$ points per coordinate (so $n=m^d$), then every $x\in[0,1]^d$ lies within Euclidean distance at most $\frac{\sqrt d}{2m}$ of some grid point. Hence
\[
h_{Z_n}\le\frac{\sqrt d}{2m}=\frac{\sqrt d}{2}\,n^{-1/d},
\]
and by \Cref{thm:filldistancebound} we obtain the explicit bound
\begin{equation*}
    \|R_n\|_{\infty}\le4Lh_{Z_n}\le2L\sqrt d\,n^{-1/d}.
\end{equation*}
More generally, for any quasi-uniform set $Z_n$ with bounded mesh ratio, one has $h_{Z_n}\le C_d\,n^{-1/d}$, yielding $\|R_n\|_{\infty}\le 4LC_d\,n^{-1/d}$ with a constant $C_d$ depending only on $d$ (and the mesh ratio). 

The downside to uniform pivoting is that it does not exploit the properties of $K$, and its empirical convergence rate does not always closely match the decay rate of the kernel's singular values. In the next section, we show that the Cholesky algorithm with complete pivoting also ensures that the fill distance goes to $0$ while also converging much faster in practice.

\section{On the convergence of the Cholesky algorithm with complete pivoting}\label{sec:proofstep2}
Now, we focus on the Cholesky algorithm with complete pivoting. Since complete pivoting always selects the next pivot as the absolute maximum of the residual, one can show that two pivots cannot be close unless the residual is already small. (To get a sense of the pivot locations selected by the Cholesky algorithm with complete pivoting, see~\Cref{fig:pretty}.)

\begin{figure} 
\begin{minipage}{.32\textwidth}
\includegraphics[width=\textwidth]{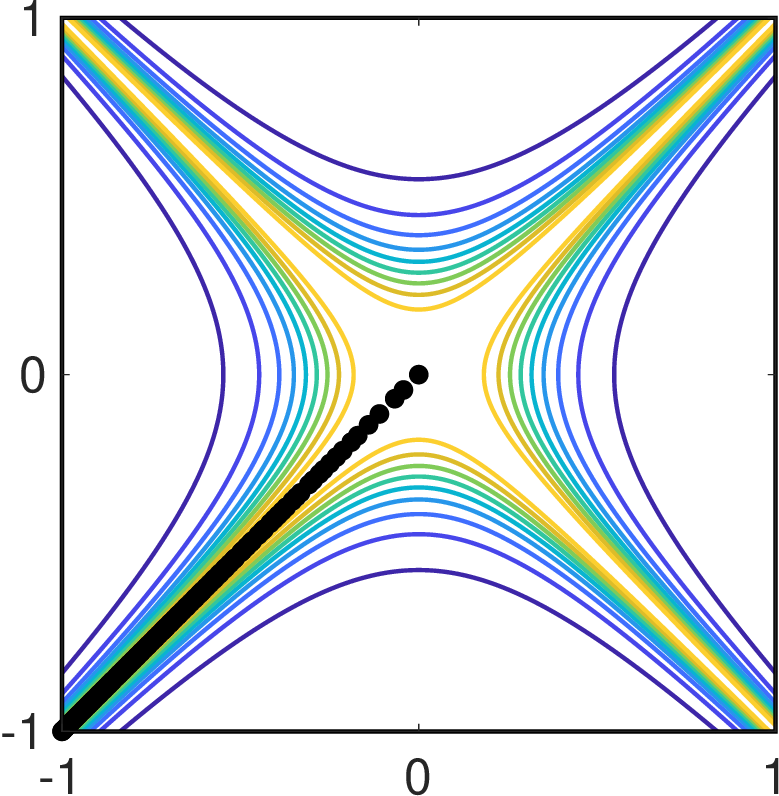}
\end{minipage} 
\begin{minipage}{.32\textwidth}
\includegraphics[width=\textwidth]{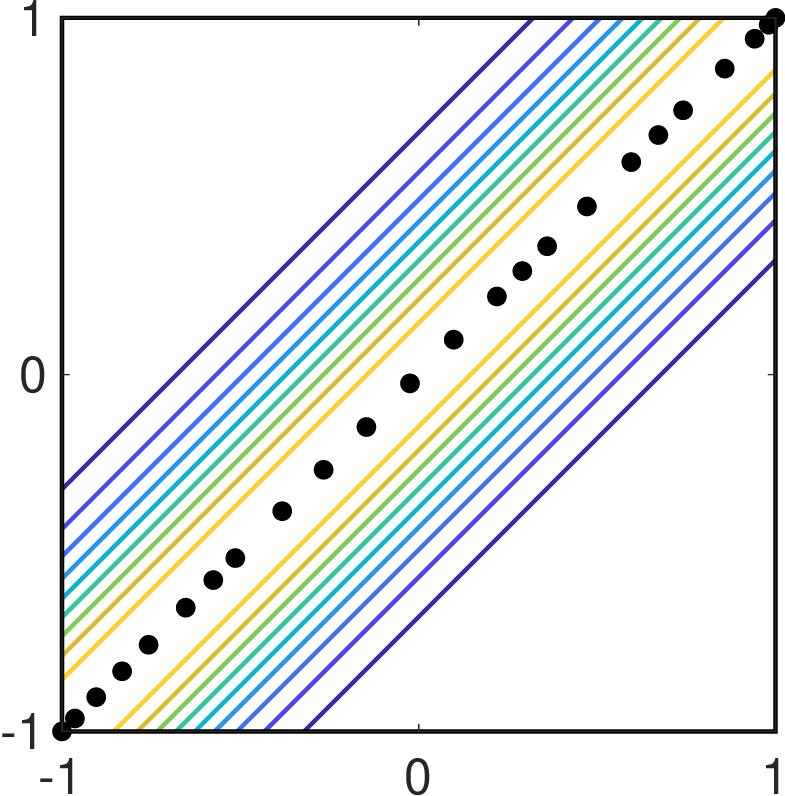}
\end{minipage} 
\begin{minipage}{.32\textwidth}
\includegraphics[width=\textwidth]{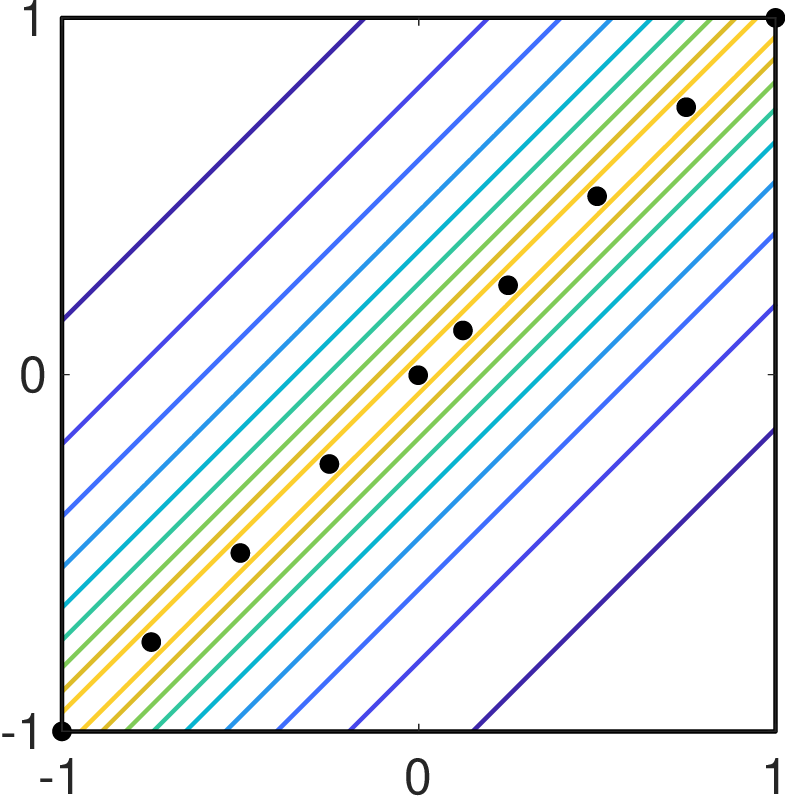}
\end{minipage} 
\begin{minipage}{.32\textwidth}
\includegraphics[width=\textwidth]{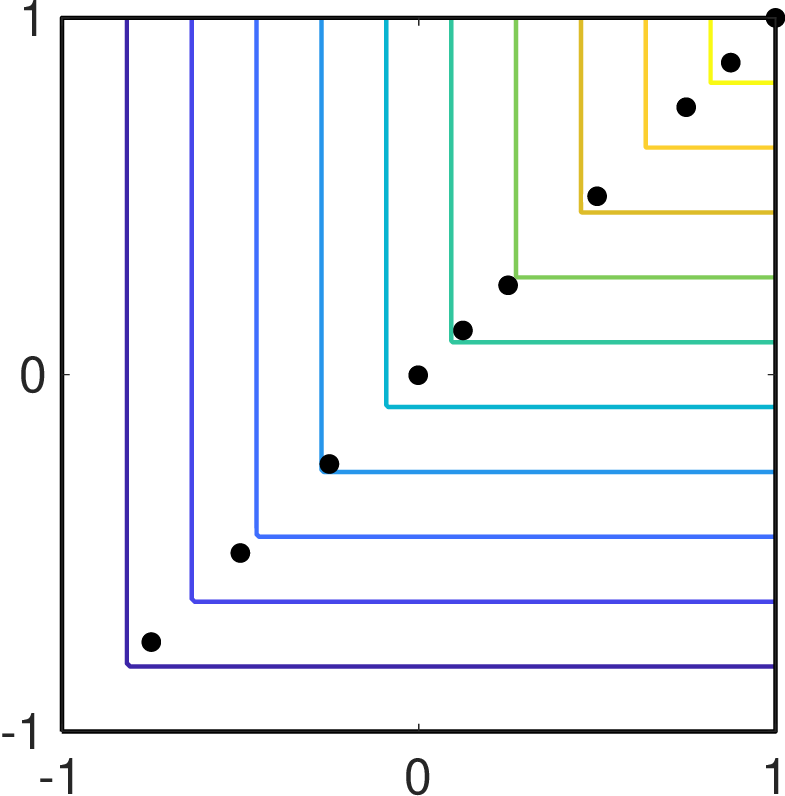}
\end{minipage} 
\begin{minipage}{.32\textwidth}
\includegraphics[width=\textwidth]{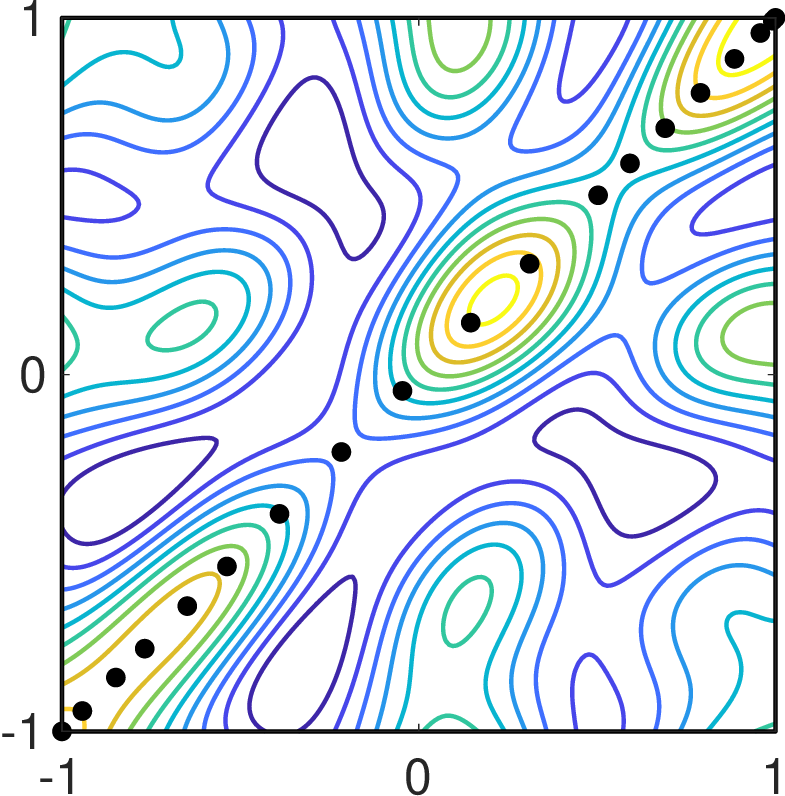}
\end{minipage} 
\begin{minipage}{.32\textwidth}
\includegraphics[width=\textwidth]{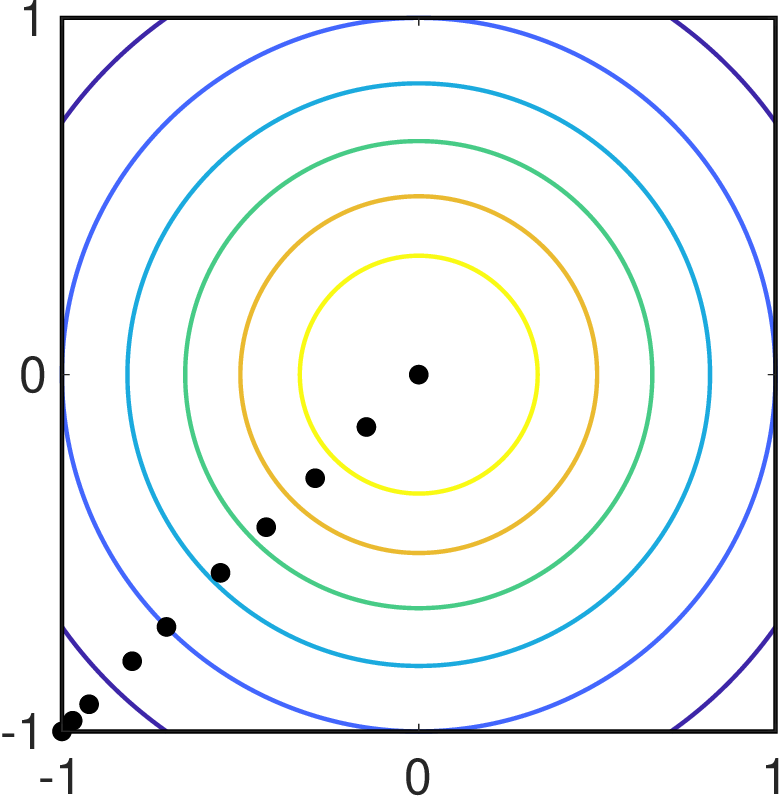}
\end{minipage} 
\caption{Six contour plots of kernels with $\Omega = [-1,1]$, showing the pivot location (black dots) of the Cholesky algorithm with complete pivoting until the residual is below machine precision. \emph{Top left:} $K(x,y) = 1/(1+100(x^2-y^2)^2)$. \emph{Top center:} $K(x,y) = e^{-5(x-y)^2}$. \emph{Top right:} Mat\'{e}rn kernel with $\nu = 1/2$ and $\ell=1/2$ (only showing first 10 pivots). \emph{Bottom left:} Wiener kernel $K(x,y) = \min\{x,y\}+1$ (only showing first 10 pivots). \emph{Bottom center:} Kernel from~\cite[Fig.~4.10 (right)]{townsend2014computing}. \emph{Bottom right:} $K(x,y) = 1/(1 + x^2 + y^2)$.}
\label{fig:pretty}
\end{figure}

\begin{lemma}\label{lem:pivotdistance}
Let $K:\Omega\times \Omega\to\mathbb{R}$ be a SPD kernel satisfying~\cref{eq:LipschitzAlongDiagonal}. The nth residual $R_n$ of the Cholesky algorithm with complete pivoting satisfies 
\[
\|R_n\|_\infty \leq 4Lq_{Z_n},
\]
where $Z_n=\{z_1,\dots,z_n\}$ is the set of pivots and $q_{Z_n}$ is defined in~\cref{eq:sep-dist}.
\end{lemma}
\begin{proof}
Let $z_i$ and $z_j$ be the closest two pivots with $i<j$. From~\Cref{lem:MultiStepLipschitzControl} with $n = j-1$, we see that 
\[
\|R_{j-1}\|_\infty = |R_{j-1}(z_j,z_j)| \leq 4L \|z_i - z_j\|
\]
where the first equality is by definition of complete pivoting. The result then follows because $\|R_0\|_\infty \geq \|R_1\|_\infty \geq \|R_2\|_\infty\geq \cdots$, and by definition of $q_{Z_n}$.
\end{proof}

Given~\Cref{lem:pivotdistance} and a given number of pivots $n$, all we have to do is find a number $r$ such that there are always two pivots $z_i, z_j$ with $i, j\leq n$ such that $\|z_i - z_j\| \leq r$. 

\begin{theorem}\label{thm:convergencegeneral}
    Let $\Omega \subset \mathbb{R}^d$ be a compact domain contained in a ball of radius $R$. Let $K:\Omega\times \Omega\to\mathbb{R}$ be a kernel that satisfies the assumptions of~\Cref{thm:filldistancebound}. Then, the $n$th residual $R_n$, $n>1$, of the Cholesky algorithm with complete pivoting satisfies
    \begin{equation}\label{eq:generalbound}
        \|R_n\|_{\infty} \leq \frac{8LR}{n^{1/d}-1} = \mathcal{O}(n^{-1/d}). 
    \end{equation}
\end{theorem}
\begin{proof}
    For $n=1$, from \Cref{lem:MultiStepLipschitzControl}, we directly have $\|R_1\|_\infty \leq 4L \max_{x, y\in\Omega}\|x-y\| \leq 8LR$. For $n>1$, let $r = 2R/(n^{1/d}-1)$. This $r$ is designed to satisfy the following identity exactly:
    \begin{equation*}
        n = \left(1 + \frac{2R}{r}\right)^d = \frac{\operatorname{vol}(B_{R + r/2})}{\operatorname{vol}(B_{r/2})}, 
    \end{equation*}
    where $B_r$ is the closed ball in $\mathbb{R}^d$ of radius $r$ and $\operatorname{vol}$ is its volume. Denote the first $n$ pivots of Cholesky as $z_1, \dots, z_n$ and consider $n$ closed balls of radius $\frac{r}{2}$ centered at $z_1, \dots, z_n$. It is impossible for all of these $n$ balls to be disjoint, as otherwise, 
        \begin{equation*}
        n \operatorname{vol}(B_{r/2}) < \operatorname{vol}(B_{R+\frac{r}{2}}),
    \end{equation*}
    which contradicts the fact that $n = \left(1 + \frac{2R}{r}\right)^d$. Therefore, we must have $\min_{1\leq i,j\leq n}\|z_i - z_j\|\leq r$. The bound on $\|R_n\|_\infty$ in~\cref{eq:generalbound} follows from \Cref{lem:pivotdistance}.
\end{proof}

\Cref{thm:convergencegeneral} shows that complete pivoting necessarily drives the residual $\|R_n\|_\infty$ to zero at a dimension-dependent algebraic rate $\mathcal{O}(n^{-1/d})$, regardless of the geometry of $\Omega$. The very surprising observation is that this convergence rate is obtained \emph{despite the fact that the fill distance does not always tend to zero}! For example, the top left and bottom right contours in~\Cref{fig:pretty} have their fill distances stagnate at 1 as $n$ increases, since pivots never land in the interval $[0,1]$. Therefore,~\Cref{thm:convergencegeneral} is much stronger than typical results in the RBF literature, which typically require a quasi-uniformity assumption on the pivot points to prove convergence~\cite{de2005near}.



In~\Cref{fig:ConvergenceDiffD}, we demonstrate our convergence rate for the Cholesky algorithm with complete pivoting for $d = 1, 2, 3$ with the Mat\'{e}rn kernel with $\nu=1/2$, $\ell = 1/2$, and $\Omega = [-1,1]^d$.


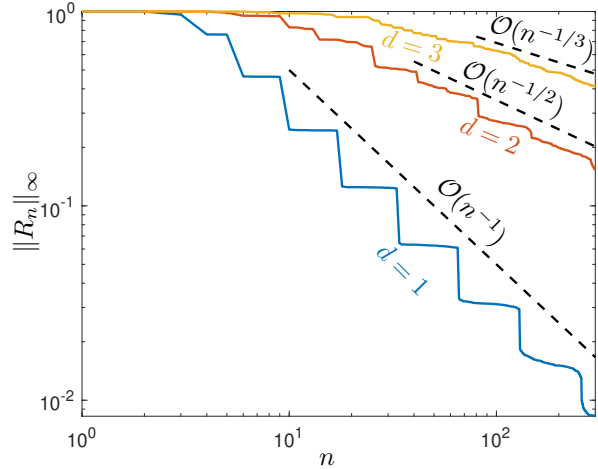
\begin{figure} 
\begin{minipage}{.49\textwidth}
\begin{overpic}[width=\textwidth]{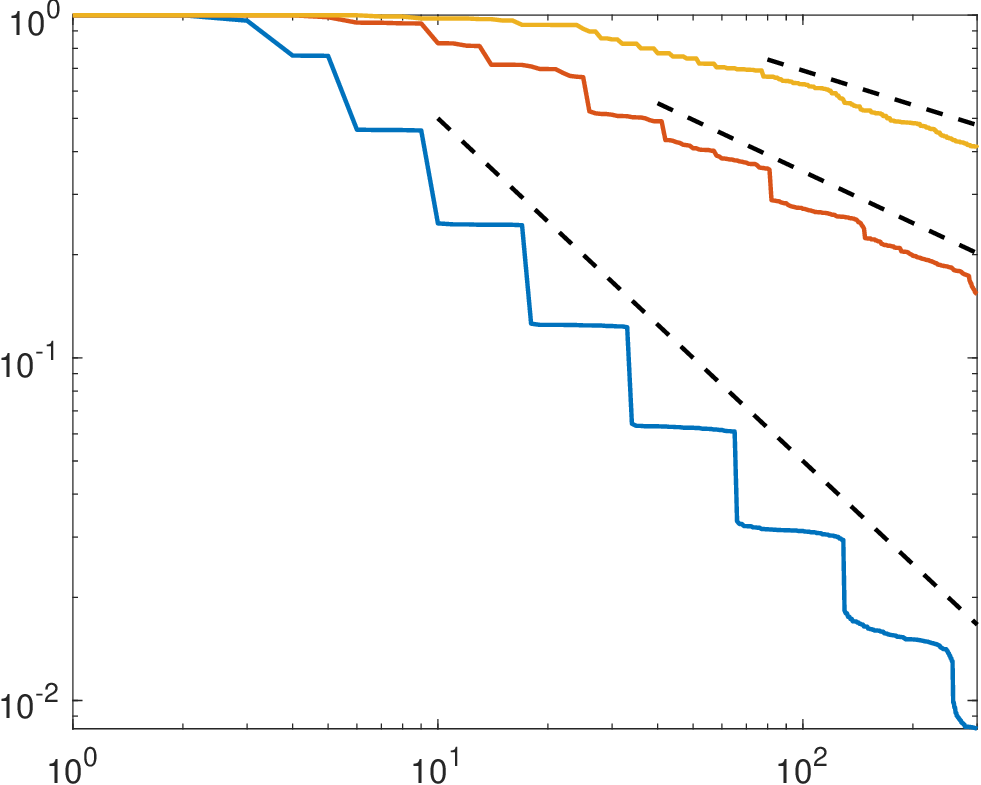}
 \put(70,47) {\rotatebox{-45}{$\mathcal{O}(n^{-1})$}}
  \put(75,67) {\rotatebox{-23}{$\mathcal{O}(n^{-1/2})$}}
   \put(80,75) {\rotatebox{-16}{$\mathcal{O}(n^{-1/3})$}}
    \put(-5,35) {\rotatebox{90}{$\|R_n\|_\infty$}}
  \put(60,35) {\color{matlabblue}\rotatebox{-45}{$d = 1$}}
  \put(75,57) {\color{matlabred}\rotatebox{-23}{$d = 2$}}
   \put(61,72.5) {\color{matlabyellow}\rotatebox{-17}{$d = 3$}}
    \put(50,-3) {$n$}
\end{overpic}
\end{minipage}
\caption{The convergence of the Cholesky algorithm with complete pivoting for the Mat\'{e}rn kernel with $\nu = 1/2$, $\ell = 1/2$, and $\Omega = [-1,1]^d$ for $d = 1$ (blue), $d=2$ (red), and $d=3$ (yellow). The convergence rate matches the $\mathcal{O}(n^{-1/d})$ rate predicted by~\Cref{thm:convergencegeneral}.}
\label{fig:ConvergenceDiffD}
\end{figure} 
\section{Other pivoting strategies}\label{sec:OtherPivotingStrategies}
The Cholesky algorithm can be employed with many different pivoting strategies. In this section, we briefly consider other pivoting strategies such as: (1) local maximum volume pivoting (see~\Cref{sec:LocalMaxVol}), (2) complete pivoting with mistakes (see~\Cref{sec:completePivotingWithMistakes}), and (3) discrete pivoting (see~\Cref{sec:discretepivoting}).

\subsection{Global and local maximum volume pivoting}\label{sec:LocalMaxVol}

Global and local maximum volume pivoting are classical strategies in numerical linear algebra for selecting pivots in rank-revealing factorizations~\cite{damle2025estimating,gu1996efficient}. Given a set of previously chosen pivots, the idea is to select the next pivot to maximize the determinant (i.e., the volume) of the corresponding pivot submatrix. That is, the first $k$ pivots $Z_{k} = \left\{z_1,\ldots,z_k\right\}$ for global maximum volume pivoting are chosen so that
   \[
    \det\bigl(K[Z_k,Z_k]\bigr) \geq \det\bigl(K[B_k,B_k]\bigr)
    \]
    where $K[Z_k,Z_k]$ is the $k\times k$ submatrix of $K$ obtained by evaluating the kernel at $Z_k$ forming a matrix with entries $K(z_i,z_j)$ and $B_k$ is any collection of $k$ points in $\Omega$. Despite its beautiful theoretical properties, global maximum volume is usually not computationally feasible~\cite{civril2009selecting}. Local maximum volume is a more feasible alternative, where the first $k$ pivots $Z_{k+1}$ are selected so that 
    \[
    \det\bigl(K[Z_{k},Z_{k}]\bigr) \geq \det\bigl(K[\left(Z_{k}\setminus z\right)\cup\{y\},\left(Z_{k}\setminus z\right)\cup\{y\}]\bigr),
    \]
    for all $z\in Z_{k}$ and all $y\in \Omega \setminus Z_k$. That is, the next pivot is chosen so that replacing any one pivot with the new candidate would not increase the volume of the selected submatrix. Both global and local maximum volume pivoting are non-greedy pivot selection criteria. The same error bound as in~\Cref{thm:convergencegeneral} holds for the Cholesky algorithm with global or local maximum volume pivoting. 

\begin{corollary}
    With the same assumptions as \Cref{thm:filldistancebound}, the $n$th residual $R_n$ of the Cholesky algorithm with global or local maximum volume pivoting satisfies
    \begin{equation*}
        \|R_n\|_\infty \leq 4Lq_{Z_n}\le \frac{8LR}{n^{1/d}-1},
    \end{equation*}
    where $Z_n=\{z_1,\dots,z_n\}$ is the set of pivots and $q_{Z_n}$ is defined in~\eqref{eq:sep-dist}.
    \label{cor:maxvol}
\end{corollary}
\begin{proof}
    We will once more prove that the residual satisfies the bound $\|R_n\|_\infty \leq 4L \min_{i\neq j}\|z_i - z_j\|$. For a given $n$, let $z_{\ell}, z_k$ be the two pivots with the minimum distance. Consider the residual $\tilde{R}_{n-1}$ obtained by performing $n-1$ Cholesky steps with pivots $z_1, \dots, z_{k-1},z_{k+1},\ldots, z_n$. Define the set $S_k = \{1, \dots, n\}\setminus\{k\}$. The diagonal values of $\tilde{R}_{n-1}$ are given by    \begin{equation*}
        \tilde{R}_{n-1}(x, x) = K(x, x) - \left([K(x, z_i)]_{i=S_k}\right)^T [K(z_i, z_j)]_{i, j\in S_k} [K(x, z_i)]_{i=S_k},
    \end{equation*}
    and notice that 
    \begin{equation*}
        \det\left(K[S_k \cup \{y\}, S_k \cup \{y\}]\right) = \det\left(K[S_k,S_k]\right) \cdot \tilde{R}_{n-1}(y, y),
    \end{equation*}
    for all $y\in \Omega$, from the block determinant formula. If there is $y\in\Omega$ such that $\tilde{R}_{n-1}(y, y)>\tilde{R}_{n-1}(z_k, z_k)$, then the volume corresponding to pivots $S_k \cup \{y\}$ is larger than $S_k\cup\{z_k\} = \{z_1, \dots, z_n\}$, which is a contradiction as it would violate the global or local maximum volume criterion. Thus, the diagonal of $\tilde{R}_{n-1}$ must attain its maximum at $z_k$. By \Cref{lem:MultiStepLipschitzControl} and the fact that $z_{\ell}$ is the nearest pivot to $z_k$, we have that $R_n(x, x) \leq \tilde{R}_{n-1}(x, x) \leq \tilde{R}_{n-1}(z_k, z_k)\leq 4L \|z_\ell - z_k\|$. 
\end{proof}
This corollary demonstrates that maximum volume pivoting achieves the same convergence rate as complete pivoting. However, it is essential to note that this bound is not rank-revealing, which one might hope for. While~\Cref{cor:maxvol} controls the decay of the residual uniformly, it does not guarantee that the pivot set exposes the numerical rank of the kernel matrix in the sense of strong rank-revealing factorizations~\cite{damle2025estimating}. An alternative analysis is needed to address rank-revealing properties of maximum volume pivoting.

\subsection{Complete pivoting with mistakes}\label{sec:completePivotingWithMistakes}
In software implementations, complete pivoting may not be practical because of the requirement to compute the maximum of the diagonal of the residual at every step. To resolve such a problem, one can approximately select the maximum at each step, without significantly degrading the computed approximant.
Let $0<\delta\leq 1$ be a parameter measuring the quality of a pivot. We say that $\delta$-complete pivoting is the scheme where at the $i$th Cholesky step, one selects any pivot point $z_i$ such that 
\[
R_{i-1}(z_i, z_i) \geq \delta \| R_{i-1}\|_{\infty}.
\]
For instance, when $\delta=1/10$, any point $x\in\Omega$ may be chosen as a pivot provided that $R_{i-1}(x,x)$ is at least one-tenth of the maximum diagonal value. This relaxation of complete pivoting can make finding pivots significantly easier. Moreover,~\Cref{thm:convergencegeneral} can be easily extended to $\delta$-complete pivoting to obtain: 

\begin{corollary}
    Let $0< \delta \leq 1$. With the same assumptions as~\Cref{thm:filldistancebound}, the $n$th residual $R_n$ of the Cholesky algorithm with $\delta$-complete pivoting satisfies
    \begin{equation*}
        \|R_n\|_{\infty} \leq \frac{8 LR}{\delta(n^{1/d}-1)}.
    \end{equation*}
\end{corollary}

\subsection{Discrete pivoting}\label{sec:discretepivoting}

In practice, exact complete pivoting is often impractical: at each step, one would have to evaluate the diagonal of the residual everywhere in $\Omega$ and locate its global maximum. A standard remedy is to discretize the diagonal search: fix an $\eta$-net (grid) $G\subset\Omega$ and, at each step, pick the point in $G$ with the largest diagonal residual.

For kernels that are Lipschitz on the diagonal, our one-sided control gives (see~\Cref{lem:MultiStepLipschitzControl})
\[
0\le R_{i-1}(x,x)\le4L\,\mathrm{dist}(x,Z_{i-1}),
\]
where $\mathrm{dist}(x,Z_{i-1})$ is the minimum distance between $x$ and a pivot location in $Z_{i-1}$. However, this bound is not enough to quantify how well a grid captures large residuals. Let $M=\|R_{i-1}\|_\infty$ and let $x_\star$ be a maximizer, i.e., $R_{i-1}(x_\star,x_\star)=M$. If $y\in G$ is any grid point with $\|x_\star-y\|\le\eta$, then
\[
\mathrm{dist}(x_\star,Z_{i-1})
\le\|x_\star-y\|+\mathrm{dist}(y,Z_{i-1})
\le\eta+\mathrm{dist}(y,Z_{i-1}).
\]
Multiplying by $4L$ and using $M\le 4L\,\mathrm{dist}(x_\star,Z_{i-1})$ and $R_{i-1}(y,y)\le 4L\,\mathrm{dist}(y,Z_{i-1})$ gives the additive capture bound
\[
M \le 4L\eta + R_{i-1}(y,y).
\]
Hence, the grid maximizer $y_G\in G$ satisfies
\[
R_{i-1}(y_G,y_G)\ge M-4L\eta.
\]
In particular, whenever $4L\eta\le (1-\delta)M$, the grid pick is a \(\delta\)-approximate pivot:
\[
R_{i-1}(y_G,y_G) \ge \delta\|R_{i-1}\|_\infty.
\]
Thus, discretizing the search reduces exact complete pivoting to $\delta$-complete pivoting with a tolerance that is explicitly controlled by the grid spacing~$\eta$. As the algorithm progresses and the maximum residual $M$ shrinks, one may refine the grid (decrease $\eta$) to maintain a target~$\delta$ bounded away from~$0$.

If at every step we select a pivot satisfying 
\(
R_{i-1}(z_i,z_i)\ge \delta\|R_{i-1}\|_\infty
\)
(with any $0<\delta\le 1$), we find that 
\(
\|R_{j-1}\|_\infty \le (4L/\delta)\,\min_{i<j}\|z_i-z_j\|
\),
and therefore
\[
\|R_n\|_\infty \le \frac{8 L R}{\delta\,(n^{1/d}-1)}.
\]
Thus, choosing a grid with spacing $\eta$ such that $4L\eta\le (1-\delta)\|R_{i-1}\|_\infty$ guarantees that the grid-based pivot achieves at most a $\delta$ pivoting mistake at step $i$. 

In practice, a simple strategy is to start with a coarse grid and refine it adaptively so that $\delta$ remains near~$1$ while keeping evaluations inexpensive.  There is an adaptive pivoting strategy that one can use for Gaussian elimination on general bivariate functions, and one can employ a similar discrete pivoting strategy for SPD kernels (see~\cite{townsend2013extension}). 

\section{Error bounds for the Cholesky algorithm on kernels with extra smoothness}\label{sec:generalsmooth}

We now turn to kernels with additional smoothness and show that stronger convergence rates can be obtained, assuming the domain is convex. We write $\partial^j_1K(x,y)$ for the $j$th component of the gradient of $K(x,y)$ with respect to its first variable, and similarly $\partial_2^{j}K(x,y)$ for derivatives with respect to the second variable. For a scalar function $f:\Omega\to\mathbb{R}$, we use $\partial^j f$ to denote the $j$th component of its gradient.

\begin{lemma}\label{lem:lipschitzchangemulti}
Let $\Omega\subset\mathbb{R}^d$ be an open, bounded domain and let $K:\Omega\times\Omega\to\mathbb{R}$ be a continuous SPD kernel whose first partial derivatives exist everywhere.
Assume:
\begin{enumerate}[label=(\alph*)]
\item There exists $L>0$ such that, for all $1\le j\le d$,
\begin{equation}\label{eq:diag-deriv-lip}
\bigl|\partial_1^{j}K(x,y)-\partial_1^{j}K(x,x)\bigr|\le L\|x-y\|,
\qquad
\bigl|\partial_2^{j}K(x,y)-\partial_2^{j}K(x,x)\bigr|\le L\|x-y\|,
\quad x,y\in\Omega.
\end{equation}
\item There exists $L_0>0$ such that
\begin{equation}\label{eq:first-arg-lip}
\bigl|K(x,y)-K(x',y)\bigr|\le L_0\|x-x'\|\qquad \forall x,x',y\in\Omega.
\end{equation}
\end{enumerate}
Let $R_1$ be the residual after one Cholesky step with any pivoting strategy and let $z_1$ be the first pivot. Then the diagonal of the residual $D_1(x)=R_1(x,x)$ is differentiable on $\Omega$, and for each $j=1,\dots,d$,
\begin{equation}\label{eq:grad-D1-linear}
\bigl|\partial^j D_1(x)\bigr|\le \left(4L + \frac{2L_0^2}{K(z_1,z_1)}\right)\,\|x-z_1\|\qquad \forall x\in\Omega.
\end{equation}
\end{lemma}
\begin{proof}
Write
\[
D_1(x)=K(x,x)-\frac{K(x,z_1)^2}{K(z_1,z_1)}
\]
and 
\[
\partial^j D_1(x)
= \partial^j_1K(x,x)+\partial^j_2K(x,x)
-\frac{2}{K(z_1,z_1)}\,K(x,z_1)\,\partial^j_1K(x,z_1).
\]
Because $R_1(\cdot,z_1)=R_1(z_1,\cdot)=0$, we have $\partial^j D_1(z_1)=0$ for all $j$.
Subtracting $\partial^j D_1(z_1)$ and using the triangle inequality,
\begin{align*}
|\partial^j D_1(x)|
&= \bigl|\partial^j D_1(x)-\partial^j D_1(z_1)\bigr| \\
&\le \underbrace{\bigl|\partial^j_1K(x,x)-\partial^j_1K(x,z_1)\bigr|}_{\le\, L\|x-z_1\|\ \text{by }\eqref{eq:diag-deriv-lip}}
+\underbrace{\bigl|\partial^j_2K(x,x)-\partial^j_2K(x,z_1)\bigr|}_{\le\, L\|x-z_1\|\ \text{by }\eqref{eq:diag-deriv-lip}} \\
&\quad + \frac{2}{K(z_1,z_1)}
\underbrace{\bigl|K(x,z_1)-K(z_1,z_1)\bigr|}_{\le\, L_0\|x-z_1\|\ \text{by }\eqref{eq:first-arg-lip}}
\ \underbrace{\bigl|\partial^j_1K(x,z_1)\bigr|}_{\le\, L_0\ \text{since }K\text{ is }L_0\text{-Lipschitz in }x} \\
&\quad + \frac{2}{K(z_1,z_1)}\,K(z_1,z_1)\,
\underbrace{\bigl|\partial^j_1K(x,z_1)-\partial^j_1K(z_1,z_1)\bigr|}_{\le\, L\|x-z_1\|\ \text{by }\eqref{eq:diag-deriv-lip}}.
\end{align*}
Here we used that \cref{eq:first-arg-lip} and differentiability in $x$ imply
\(
\sup_{(x,y)\in\Omega^2}|\partial_1^{j}K(x,y)|\le L_0
\)
(by the mean value theorem). Combining the bounds gives
\[
|\partial^j D_1(x)| \le \left(4L + \frac{2L_0^2}{K(z_1,z_1)}\right)\,\|x-z_1\|,
\]
which proves \cref{eq:grad-D1-linear}.
\end{proof}

The assumptions of~\Cref{lem:lipschitzchangemulti} are automatically satisfied by $C^{1,1}$ kernels.  Indeed, if $K\in C^{1,1}(\Omega\times\Omega)$, then all first partial derivatives of $K$ are globally Lipschitz on the compact set $\Omega\times\Omega$.  This immediately implies condition~\cref{eq:diag-deriv-lip}, since fixing $x$ and comparing $\partial^j_1 K(x,y)$ with $\partial^j_1 K(x,x)$ amounts to moving along a line of length $\|x-y\|$ in the $(x,y)$ variables.  Likewise, continuity and boundedness of $\nabla_x K$ on $\Omega\times\Omega$ ensure condition~\cref{eq:first-arg-lip} via the mean value theorem, with constant $L_0=\sup_{(x,y)\in\Omega^2}\|\nabla_x K(x,y)\|$. 
Thus, every $C^{1,1}$ kernel satisfies both (a) and (b), so the lemma applies in particular to this class of kernels. 

\begin{lemma}\label{lem:boundedregiondiff2}
Assume the hypotheses of \Cref{lem:lipschitzchangemulti}. Further assume that $\Omega$ is convex and that $K_{\min} = \min_{z\in\Omega} K(z,z)>0$. Then the $n$th residual $R_n$ of the Cholesky algorithm with any pivoting strategy satisfies
\begin{equation}\label{eq:quadratic-near-pivots}
0\le R_n(x,x)\le \sqrt{d}\left(2L + \frac{L_0^2}{K_{\min}}\right)\,\min_{1\leq i\leq n} \|x - z_i\|^2, \qquad x\in\Omega,
\end{equation}
where $z_1,\dots,z_n$ are the Cholesky pivots.
\end{lemma}
\begin{proof}
Fix $x\in\Omega$ and choose an index $i$ so that $\|x-z_i\|$ is minimized. 
By the order-invariance of the residual with respect to the set of chosen pivots, we may perform the first Cholesky step at $z_i$ and then apply the remaining $n-1$ steps in any order to obtain the same $R_n$. Let $\tilde{R}_1$ denote the residual after the single step at pivot $z_i$, and set $\tilde{D}_1(x)=\tilde{R}_1(x,x)$.  By~\Cref{lem:lipschitzchangemulti}, each partial derivative of $\tilde{D}_1$ satisfies the pointwise bound
\[
|\partial^j \tilde{D}_1(x)| \le \underbrace{\left(4L + \frac{2L_0^2}{K(z_i,z_i)}\right)}_{=C_1}\,\|x-z_i\|.
\]
Hence, $\|\nabla \tilde{D}_1(x)\|\le \sqrt{d}\,C_1\,\|x-z_i\|$. 
Using $\tilde{D}_1(z_i)=0$ and $\nabla \tilde{D}_1(z_i)=0$, integrate along the line segment $\gamma(t)=z_i+t(x-z_i)$, we find that
\begin{multline*}
\tilde{D}_1(x)
= \int_0^1 \nabla \tilde{D}_1(\gamma(t))\cdot (x-z_i)\,dt
\le \int_0^1 \|\nabla \tilde{D}_1(\gamma(t))\|\,\|x-z_i\|\,dt \\
\le\int_0^1 \sqrt{d}\,C_1\,t\,\|x-z_i\|^2\,dt
=\frac{\sqrt{d}\,C_1}{2}\,\|x-z_i\|^2.
\end{multline*}
Since subsequent Cholesky steps can only decrease the diagonal of the residual,
\[
0 \le R_n(x,x) \le \tilde{D}_1(x) \le \frac{\sqrt{d}\,C_1}{2}\,\|x-z_i\|^2.
\]
The argument can be repeated for any $x\in\Omega$ to reach the conclusion in~\cref{eq:quadratic-near-pivots}. 
\end{proof}

In particular,~\Cref{lem:boundedregiondiff2} shows that for $C^{1,1}$ kernels the diagonal of the residual is locally controlled at second order near the pivot set. For $C^{1,1}$ kernels, we see that the residual vanishes at each pivot and cannot grow faster than quadratically with the distance to the nearest pivot. 
This quadratic control is a marked improvement over the merely linear behavior available under diagonal Lipschitz continuity. We can write the bound in terms of the fill distance: 
\begin{theorem}\label{thm:filldistance-squared}
Assume the hypotheses of \Cref{lem:boundedregiondiff2}. Then, for all $x\in\Omega$,
\[
0\le R_n(x,x)\le \sqrt{d}\left(2L + \frac{L_0^2}{K_{\min}}\right) h^2_{Z_n},
\]
where $h_{Z_n}$ is the fill distance of $\{z_1,\dots,z_n\}$. 
\end{theorem}
\begin{proof}
The result is immediate from~\Cref{lem:boundedregiondiff2}. 
\end{proof} 

The fact that we can bound $R_n$ on the diagonal quadratically with the distance to the nearest pivot is the key mechanism behind an improved convergence rate for both uniform pivoting and complete pivoting. 

For uniform pivoting, the quadratic diagonal control in \Cref{thm:filldistance-squared} leads to a substantially faster convergence rate. 
As before, suppose $\Omega=[0,1]^d$ and take $Z_n$ to be a tensor grid with $m$ points per coordinate (so $n=m^d$). Then, we have $h_{Z_n}\leq \frac{\sqrt d}{2}\,n^{-1/d}$ and hence,
\[
\|R_n\|_{\infty}\le\sqrt{d}\left(2L+\frac{L_0^2}{K_{\min}}\right)\,h_{Z_n}^2 \le Cd^{3/2}\,n^{-2/d},
\]
for some constant $C>0$. Thus, under $C^{1,1}$ regularity of the kernel $K$, uniform pivoting improves the algebraic rate of convergence $\mathcal{O}(n^{-2/d})$ as $n\rightarrow \infty$.

For complete pivoting, the quadratic diagonal control again yields an improved rate. Using the same argument as~\Cref{lem:pivotdistance}, we find that
\[
\|R_n\|_{\infty}\le\sqrt{d}\left(2L + \frac{L_0^2}{K_{\min}}\right) \left(\frac{2R}{n^{1/d} - 1}\right)^2=\mathcal{O}(n^{-2/d}),
\]
where $\Omega$ is contained in a ball of radius $R$. Thus, under $C^{1,1}$ regularity of the kernel, complete pivoting achieves the same $\mathcal{O}(n^{-2/d})$ rate as uniform pivoting, but without requiring predetermined pivoting locations. In~\Cref{fig:smoother}, we show that as the kernel gets smoother, we witness higher rates of convergence of the Cholesky algorithm with complete pivoting. 

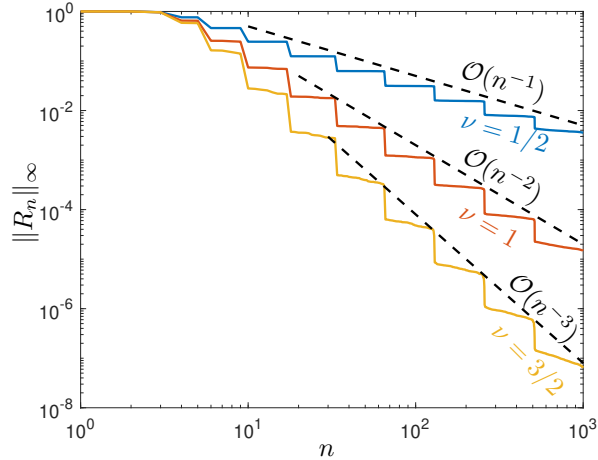
\begin{figure}
\begin{minipage}{.49\textwidth}
\begin{overpic}[width=\textwidth]{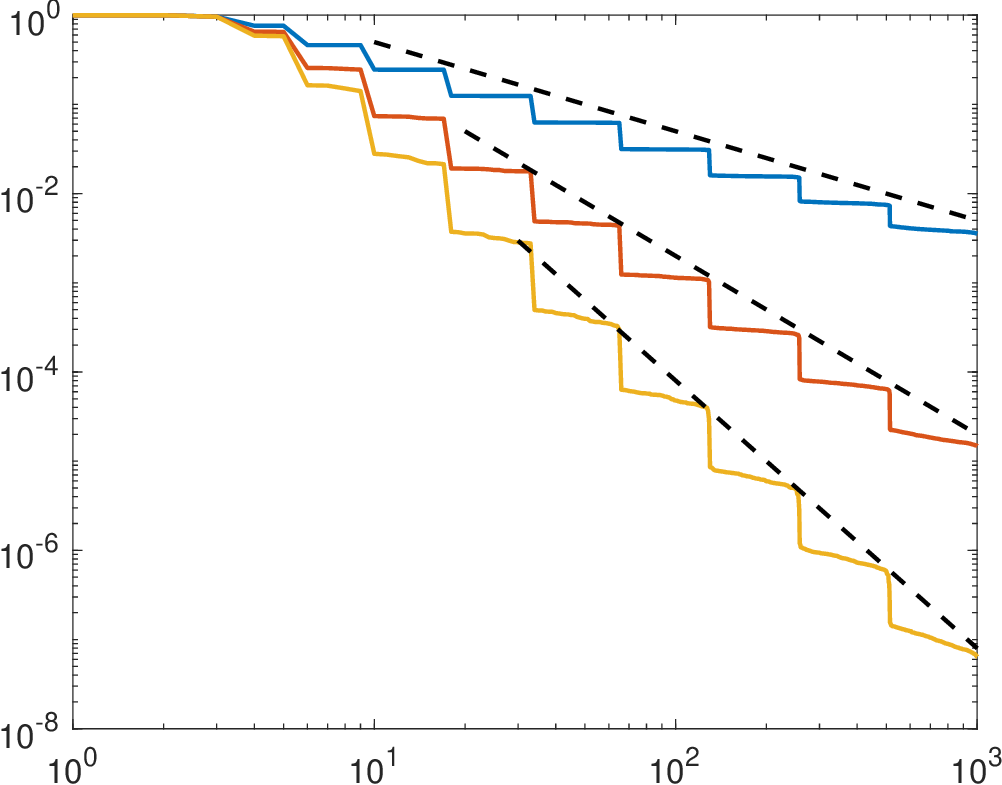}
  \put(83,28) {\rotatebox{-45}{$\mathcal{O}(n^{-3})$}}
  \put(75,50) {\rotatebox{-30}{$\mathcal{O}(n^{-2})$}}
  \put(75,65) {\rotatebox{-16}{$\mathcal{O}(n^{-1})$}}
  \put(80,19) {\color{matlabyellow}{\rotatebox{-45}{$\nu=3/2$}}}
  \put(75,40) {\color{matlabred}{\rotatebox{-30}{$\nu=1$}}}
  \put(75,56) {\color{matlabblue}{\rotatebox{-16}{$\nu=1/2$}}}
  \put(-5,35) {\rotatebox{90}{$\|R_n\|_\infty$}}
  \put(50,-3) {$n$}
\end{overpic}
\end{minipage}
\caption{The convergence of the Cholesky algorithm with complete pivoting for the Mat\'{e}rn kernel with $\ell=1/2$, $\Omega = [-1,1]$, and $d = 1$ for $\nu = 1/2$ (blue), $\nu=1$ (red), and $\nu = 3/2$ (yellow). The theory provided in this paper does not explain the $\mathcal{O}(n^{-3})$ convergence rate when $\nu=3/2$. However, since the Mat\'{e}rn kernel with $\nu\geq 3/2$ is translation-invariant and $C^2$, the P-greedy literature is a useful resource (see~\Cref{sec:Pgreedy} and~\cite{santin6convergence}).}
\label{fig:smoother}
\end{figure} 

We remark that a similar argument for kernels that are several times differentiable fails to yield a higher-order convergence rate. The following example illustrates why. Let $\Omega=[0,1]$ and consider the Gaussian kernel 
\[
K(x,y) = \exp\!\Big(-\frac{(x-y)^2}{2\sigma^2}\Big), \qquad \sigma>0,
\]
which is infinitely many times differentiable. Fix a pivot $z\in[0,1]$ and perform a single Cholesky step at $z$. On the diagonal, the residual is
\[
R_1(x,x) = 1 - \frac{K(x,z)^2}{K(z,z)}
= 1 - \exp\!\Big(-\frac{(x-z)^2}{\sigma^2}\Big).
\]
A Taylor expansion of the exponential at $x=z$ gives
\[
\frac{(x-z)^2}{\sigma^2} - \frac{(x-z)^4}{2\sigma^4}
\le R_1(x,x) \le \frac{(x-z)^2}{\sigma^2},
\]
so there exist constants $c_\sigma,r_\sigma>0$ such that
\[
R_1(x,x) \ge c_\sigma\,|x-z|^2 \qquad\text{whenever } |x-z|\le r_\sigma.
\]
Thus, even with extra regularity, we only have a quadratic bound on the residual around pivot points; for higher-order bounds, we would need the derivatives of the Cholesky residual to vanish at pivot points as well, which is not guaranteed.



\section{Consequences of our convergence rates}\label{sec:connections}

The pivoted Cholesky algorithm appears in a number of applications. However, in most applications, the theoretical convergence analysis is often omitted, since the algorithm performs very well in practice. Our convergence rates have several consequences for the pivoted Cholesky algorithm applied to the convergence of the power function in the P-greedy algorithm (see~\Cref{sec:Pgreedy}), to Cholesky factorization for matrices (see~\Cref{sec:matrixcounterparts}), and to Gaussian process regression (see~\Cref{sec:GaussianProcessRegression}). We briefly discuss a few of these connections now. 

\subsection{The P-greedy algorithm in the kernel interpolation literature}\label{sec:Pgreedy}

The P-greedy algorithm is a widely used greedy strategy in kernel approximation and reduced basis methods~\cite{wenzel2023analysis}. Its selection rule is defined in terms of the so-called power function, which quantifies the worst-case interpolation error in the reproducing kernel Hilbert space (RKHS) associated with a kernel~\cite{de2005near,schaback2000adaptive}. 
We briefly recall these notions and then show that the P-greedy algorithm is equivalent to Cholesky with complete pivoting.

Let $\Omega\subset \mathbb{R}^d$ and $K:\Omega\times\Omega\to\mathbb{R}$ be a continuous SPD kernel, with associated RKHS $\mathcal{H}_K$. 
Given a finite set of points $Z_n=\{z_1,\dots,z_n\}\subset\Omega$, the kernel interpolant $I_{Z_n} f\in \mathcal{H}_K$ to data $\{f(z_j)\}_{j=1}^n$ is the unique function of the form
\[
(I_{Z_n} f)(x) = \sum_{j=1}^n \alpha_j K(x,z_j),
\]
that satisfies $(I_{Z_n} f)(z_i)=f(z_i)$ for $i=1,\dots,n$.  
The error of this interpolant can be bounded in terms of the power function
\[
    P_{Z_n}(x) = \sup_{\|f\|_{\mathcal{H}_K}=1}\, |f(x) - (I_{Z_n} f)(x)|, \qquad x\in\Omega,
\]
which measures the maximal interpolation error at $x$ over all unit-norm functions in $\mathcal{H}_K$.

Let $K(Z_n,Z_n)=[K(z_i,z_j)]_{i,j=1}^n$, $k(x) = (K(x,z_1),\dots,K(x,z_n))^\top$, and define the Nystr\"{o}m approximation
\[
    \widehat K_{Z_n}(x,y) = k(x)^\top K(Z_n,Z_n)^{-1}\, k(y).
\]
The residual kernel is the Schur complement
\begin{equation}\label{eq:residual-kernel}
    R_n(x,y) = K(x,y) - k(x)^\top K(Z_n,Z_n)^{-1} k(y).
\end{equation}
Note that if $Z_n=\{z_1,\dots,z_n\}$ is chosen as the set of Cholesky pivots, then $R_n$ is exactly the Cholesky residual kernel after $n$ steps.

\begin{theorem}\label{lem:diag-equals-power2}
Let $\Omega\subset\mathbb{R}^d$ be a compact domain and $K:\Omega\times \Omega \rightarrow \mathbb{R}$ a continuous SPD kernel. For every $x\in\Omega$ and every finite set of $n$ points $Z_n\subset\Omega$, we have
\[
    R_n(x,x) = \big(P_{Z_n}(x)\big)^2,
\]
where $R_n$ is the $n$th residual of the Cholesky algorithm that pivoted at $Z_n$ and $P_{Z_n}$ is the associated power function to the RKHS $\mathcal{H}_K$. 
\end{theorem}
\begin{proof}
Let $r_x(\cdot) = K(\cdot,x) - \sum_{j=1}^n \beta_j K(\cdot,z_j)$ with $\beta = K(Z_n,Z_n)^{-1} k(x)$.  
By the reproducing property of $\mathcal{H}_K$, for all $f\in \mathcal{H}_K$,
\[
f(x) - (I_{Z_n} f)(x) = \langle f, r_x \rangle_{\mathcal{H}_K},
\]
where $\langle \cdot, \cdot \rangle_{\mathcal{H}_K}$ is the $\mathcal{H}_K$ inner-product. 
Hence, by Cauchy--Schwarz and optimality,
\[
P_{Z_n}(x) = \|r_x\|_{\mathcal{H}_K}.
\]
Expanding $\|r_x\|_{\mathcal{H}_K}^2$ using the definition of $\beta$ gives
\[
\|r_x\|_{\mathcal{H}_K}^2 = K(x,x) - k(x)^\top K(Z_n,Z_n)^{-1} k(x).
\]
By~\cref{eq:residual-kernel}, this equals $R_n(x,x)$. Thus $P_{Z_n}(x)^2 = R_n(x,x)$.
\end{proof}

By~\Cref{lem:diag-equals-power2}, the squared power function at $x$ coincides with the diagonal residual $R_n(x,x)$. Therefore, the P-greedy selection rule used in the kernel literature and given by
\[
    z_{n+1} = \arg\max_{x\in \Omega} P_{Z_n}(x)
\]
is identical to the Cholesky complete pivoting rule
\[
    z_{n+1} = \arg\max_{x\in \Omega} R_n(x,x).
\]
In other words, P-greedy and the Cholesky with complete pivoting select the same pivots, and are essentially the same algorithm viewed through two different lenses: kernel approximation versus function approximation. 

Our convergence results can therefore be stated in terms of bounding the power function. In particular, for a kernel satisfying the assumptions in~\Cref{thm:convergencegeneral}, we know that
\[
    \|P_{Z_n}\|_\infty^2 = \|R_n\|_\infty \leq \frac{8LR}{n^{1/d}-1} = \mathcal{O}(n^{-1/d}), 
\]
whenever the pivot set $Z_n$ is generated by complete pivoting (equivalently, by P-greedy selection). This provides a deterministic, dimension-dependent rate for the decay of the power function under very weak smoothness assumptions.

Santin and Haasdonk~\cite{santin6convergence} analyze the P-greedy method for translation-invariant kernels, using connections to Kolmogorov $n$-widths and greedy approximation in Banach spaces. Their results yield near-optimal convergence rates depending on the smoothness of the kernel, and show that P-greedy points become asymptotically uniformly distributed in $\Omega$, hence $h_{Z_n}\to0$ as $n\to\infty$. Their result seems to depend crucially on translation invariance; note that neither of the two contours in~\Cref{fig:pretty} whose fill distance does not decay (top left and bottom right) are translation invariant. Furthermore, such results (as well as more refined estimates, e.~\cite{wenzel2021novel}) require the kernel to be sufficiently smooth, typically in at least the dimension-dependent Sobolev space $W_2^{d/2+\epsilon}(\Omega)$.

Compared with their bounds, our results apply to more general kernels (no translation invariance required), and remain valid even under low regularity (e.g., only Lipschitz continuity on the diagonal) in any dimension. On the other hand, their bounds are sharper in the high-smoothness, translation-invariant setting. Thus, the two approaches are complementary: our theory provides robust, dimension-dependent convergence guarantees under weak assumptions, while the P-greedy literature offers refined rates for smooth kernels in structured settings.

\subsection{Matrix counterparts of the kernel results}\label{sec:matrixcounterparts}

Our analysis has focused on kernel functions $K:\Omega\times\Omega\to\mathbb{R}$, where the pivoted Cholesky algorithm produces low-rank approximations with residuals controlled in terms of the fill distance of the pivot set. 
Interestingly, many of these ideas have natural discrete analogues when one replaces kernels by structured matrices. In this setting, the geometry of the domain $\Omega$ is mirrored by the index set $\{1,\dots,m\}$, and Lipschitz continuity of $K$ along the diagonal is replaced by discrete difference quotients of the matrix entries. To make this precise, let $A\in\mathbb{R}^{m\times m}$ be a SPD matrix, and define the discrete diagonal Lipschitz constant as
\begin{equation}
G_A = \max_{i\neq j}\,\frac{|A_{i,i}-A_{i,j}|}{|i-j|}.
\label{eq:GA}
\end{equation}
This quantity is the matrix analogue of the Lipschitz constant $L$ from the kernel setting: it measures how rapidly the entries of $A$ vary along each row as the column index changes. Applying the Cholesky algorithm with complete pivoting to $A$ then yields the following discrete counterpart of \Cref{thm:convergencegeneral}.

\begin{corollary}\label{cor:matrixversion}
    Let $A$ be an $m\times m$ SPD matrix and define $G_A$ as in~\cref{eq:GA}. Let $A_n$ be the rank-$n$ approximation obtained by the Cholesky decomposition with complete pivoting. 
    Then, for $1<n\le m$, we have
    \[
        \|A - A_n\|_{\max} \leq \frac{4(m-1)G_A}{n-1} =\mathcal{O}(n^{-1}),
    \]
    where $\|\cdot\|_{\max}$ denotes the maximum absolute matrix entry.
\end{corollary}

This result shows that, in the matrix setting, complete pivoting necessarily drives the residual down at a uniform $\mathcal{O}(n^{-1})$ rate. 
Just as in the kernel case, the key mechanism is that complete pivoting prevents the pivots from clustering too closely: two pivot indices cannot be adjacent unless the residual is already small. In this way, the geometry of the index set enforces a quasi-uniform distribution of pivot indices, leading to algebraic decay of the residual. 

\subsection{Gaussian process regression}\label{sec:GaussianProcessRegression}

Gaussian process regression (also known as Kriging)~\cite{williams2006gaussian} is a popular Bayesian nonparametric regression method, which is frequently used in machine learning. In the noise-free setting, one places a Gaussian process prior on the target function $f$ with mean $\mu:\Omega\to\mathbb{R}$ and covariance kernel $K:\Omega\times\Omega\to\mathbb{R}$. After $n$ observations, the posterior mean and covariance are updated via Bayes’ rule:
\[
\begin{aligned}
    \mu_n(x) &= \mu(x) + ([K(x, x_i)]_{i=1}^n)^T\left([K(x_i, x_j)]_{i, j=1}^n\right)^{-1}[y_i - \mu(x_i)]_{i=1}^n,\\
    R_n(x, x') &= K(x, x') - ([K(x, x_i)]_{i=1}^n)^T\left([K(x_i, x_j)]_{i, j=1}^n\right)^{-1}[K(x, x_i)]_{i=1}^n.
\end{aligned}
\]
Here, the posterior covariance $R_n$ coincides with the residual of the Cholesky algorithm after $n$ steps with pivot set $\{x_1,\dots,x_n\}$. Thus, bounds on the Cholesky residual directly translate into uniform variance reduction guarantees for Gaussian process regression.

Therefore, our convergence rates for the Cholesky algorithm with uniform pivoting, complete pivoting, or global/local maximum volume pivoting can be used to bound $R_n$ uniformly. In this setting, the control of the maximum variance in the Gaussian process is achieved after $n$ samples. 

Despite the empirical success of Gaussian process regression, its convergence properties are rarely addressed in the literature. The prevailing view is that smooth prior kernels yield satisfactory convergence in practice, see for instance~\cite{ambikasaran2015fast}. More systematic analyses connect Gaussian process regression to kernel interpolation, where the posterior mean corresponds to the interpolant and the posterior covariance coincides with the power function~\cite{scheuerer2013interpolation}. In this setting, deterministic convergence rates have been obtained for translation-invariant kernels~\cite{kanagawa2018gaussian,stuart2018posterior}, with few exceptions such as~\cite{de2005near}. Our results contribute to this picture by providing algebraic convergence rates for Gaussian process regression with general kernels, including those that are less smooth than $C^2$ and not necessarily translation-invariant. 

\section{Conclusion}\label{sec:conclusion}

We proved that the pivoted Cholesky algorithm converges under minimal smoothness assumptions. For SPD kernels that are Lipschitz on the diagonal, the residual is controlled by the fill distance of the pivots (see~\Cref{thm:filldistancebound}), yielding
\[
\|R_n\|_\infty=\mathcal{O}(n^{-1/d})
\]
for complete pivoting, whose residual is in fact controlled by the stronger separation distance (see~\Cref{lem:pivotdistance} and~\Cref{thm:convergencegeneral}). With $C^{1,1}$ regularity, the diagonal of the residual is locally quadratic in the distance to the nearest pivot (see~\Cref{lem:boundedregiondiff2}), giving the improved rate of
\[
\|R_n\|_\infty=\mathcal{O}(n^{-2/d})
\]
for both uniform and complete pivoting (see~\Cref{thm:filldistance-squared}). We also established a discrete analogue for SPD matrices with a diagonal Lipschitz constant $G_A$, obtaining $\|A-A_n\|_{\max}=\mathcal{O}(n^{-1})$ (see~\Cref{cor:matrixversion}).

Practically, approximate ($\delta$-complete) and grid-based pivoting retain the same algebraic rates up to constants, justifying standard implementations. The results transfer directly to Gaussian process regression (uniform variance reduction) and clarify the equivalence between complete pivoting and P-greedy via the power function. Open directions include rank-revealing refinements beyond uniform bounds, understanding staircase effects in convergence plots, and extending stability and noise-robustness guarantees.

\section*{Acknowledgments} 
A.T.~is grateful for support from NSF CAREER (DMS-2045646), DARPA-PA-24-04-07, and the Office of Naval Research under Grant Number N00014-23-1-2729. C.W.~was supported by the NSF GRFP (DGE-2139899).

\bibliographystyle{siam}
\bibliography{refs}

\end{document}